\documentclass[12pt]{article}

\usepackage[T1]{fontenc}
\usepackage{lmodern}

\usepackage{amsmath, amssymb, amsthm} 
\usepackage[margin=2.5cm]{geometry}
\usepackage{color, graphicx}

\usepackage{etoolbox}
\makeatletter
\patchcmd{\@maketitle}{\LARGE \@title}{\LARGE\bfseries\@title}{}{}

\renewcommand{\@seccntformat}[1]{\csname the#1\endcsname.\quad}
\makeatother

\definecolor{darkblue}{rgb}{0,0,.5}

\usepackage{hyperref}
\hypersetup{
	colorlinks=true,		
	linkcolor=darkblue,		
	citecolor=darkblue,		
	urlcolor=darkblue 		
}

\makeatletter
\def\th@plain{%
	\thm@notefont{}
	\itshape 
}
\def\th@definition{%
	\thm@notefont{}
	\normalfont 
}

\renewenvironment{proof}[1][\proofname]{\par
	\normalfont
	\topsep0\p@\@plus3\p@ \trivlist
	\item[\hskip\labelsep\itshape
	#1\@addpunct{.}]\ignorespaces
}{%
	\qed\endtrivlist
}
\makeatother

\newtheorem{theorem}{Theorem}[section]
\newtheorem{lemma}[theorem]{Lemma}
\newtheorem{corollary}[theorem]{Corollary}
\newtheorem{proposition}[theorem]{Proposition}

\theoremstyle{definition}
\newtheorem{definition}[theorem]{Definition}
\theoremstyle{definition}
\newtheorem{example}[theorem]{Example}
\theoremstyle{definition}
\newtheorem{remark}[theorem]{Remark}

\usepackage[shortlabels]{enumitem}

\allowdisplaybreaks

\newcommand{\R}{\ensuremath{\mathbb R}}

\newcommand{\ran}{\ensuremath{\operatorname{ran}}}
\newcommand{\zer}{\ensuremath{\operatorname{zer}}}
\newcommand{\dom}{\ensuremath{\operatorname{dom}}}

\newcommand{\gra}{\ensuremath{\operatorname{gra}}}
\newcommand{\Fix}{\ensuremath{\operatorname{Fix}}}
\newcommand{\Id}{\ensuremath{\operatorname{Id}}}

\newcommand{\sgn}{\ensuremath{\operatorname{Sign}}}

\def\beq{\begin{equation}}
\def\eeq{\end{equation}}
\def\baq{\begin{eqnarray}}
\def\eaq{\end{eqnarray}}
\def\baqn{\begin{eqnarray*}}
\def\eaqn{\end{eqnarray*}}

\if
{
\usepackage[pageref]{backref}
\renewcommand*{\backrefalt}[4]{%
\ifcase #1 %
(Not cited)%
\or
(Cited on p.~#2)%
\else
(Cited on pp.~#2)%
\fi
}

}
\fi

\begin{document}

\title{Solving Non-Monotone Inclusions Using  Monotonicity of Pairs of Operators}

\author{
Ba Khiet Le\thanks{Optimization Research Group, Faculty of Mathematics and Statistics, Ton Duc Thang University, Ho Chi Minh City, Vietnam. E-mail: \texttt{lebakhiet@tdtu.edu.vn}}, 
~
Minh N. Dao\thanks{School of Science, RMIT University, Melbourne, VIC 3000, Australia. E-mail: \texttt{minh.dao@rmit.edu.au}}, 
~and~ 
Michel Th\' era\thanks{Mathematics and Computer Science Department, University of Limoges, 123 Avenue Albert Thomas,
87060 Limoges CEDEX, France. E-mail: \texttt{michel.thera@unilim.fr}}
}

\date{\today}

\maketitle

\begin{abstract}
\noindent In this paper, under the monotonicity of pairs of operators, we propose some Generalized Proximal Point Algorithms to solve  non-monotone inclusions using warped resolvents and transformed resolvents. The weak, strong, and linear convergence of the proposed algorithms are established under very mild conditions.
\end{abstract}

\paragraph{Keywords:}
Non-monotone inclusions,  
monotonicity of pairs of operators, 
nonconvex programming,
warped resolvents, transformed resolvents,
proximal point algorithm.

\paragraph{Mathematics Subject Classification (MSC 2020):}
49J52, 49J53, 28B05.

\section{Introduction}

Our goal is to provide a new efficient method for solving the inclusion
\begin{equation}\label{main}
0 \in F(x),
\end{equation}
where \( F:\mathcal{H} \rightrightarrows \mathcal{H} \) is a set-valued operator, not necessarily monotone, defined in a Hilbert space \( \mathcal{H} \). This inclusion arises frequently in practice, particularly in nonconvex optimization, where the critical point condition involves finding
\begin{equation}\label{main2}
\min_{x \in \mathcal{H}} f(x),
\end{equation}
with \( f: \mathcal{H} \to \mathbb{R} \cup \{\infty\} \) possibly nonconvex. One of the most straightforward and effective methods for solving (\ref{main2}) is the difference-of-convex algorithm (\textbf{DCA}) (see, e.g, \cite{LeT,An,Arag,Le,Pham}), which is based on the forward-backward technique and the decomposition of \( f \) into the difference of two convex functions. Another popular approach is based on the hidden convexity and the proximal point algorithm (\textbf{PPA}) (see, e.g, \cite{Iusem,Pennanen,Rockafellar2019}). Besides, there are other well-known methods solving nonconvex problems such as gradient descent (\textbf{GD}) \cite{Nesterov2,Polyak}, accelerated gradient method \cite{Carmon}, stochastic gradient descent (\textbf{SGD}) \cite{Ge}, and cubic-regularized Newton method \cite{Nesterov1, Nesterov3}...

{Recently, Adly, Cojocaru, and Le \cite{acl} introduced a new simple   method to solve \eqref{main} with a linear convergence rate when \( F \) is a single-valued Lipschitz continuous mapping, under the strong monotonicity of the pair \( (F,v) \), where \( v: \mathcal{H} \to \mathcal{H} \) is a suitably chosen single-valued mapping. The solution set of \eqref{main} is \( F^{-1}(0) \), which involves the inverse of \( F \). Instead of computing \( F^{-1} \), which can be challenging, the method proposes finding a simpler function \( v \) similar to \( F \) (i.e., \( F \) and \( v \) share similar monotonicity properties) and for which \( v^{-1} \) is easier to compute. In this sense, \( v \) can be considered a simplified adaptation of \( F \).

The concept of  monotonicity of pairs is a natural and non-trivial extension of classical monotonicity, allowing for \( F \) to be non-monotone, particularly in cases where \( F \) has a block structure (see e.g., \cite{acl,LT} and Examples 1, 2). An intriguing question is whether this method can be applied when \( F \) is set-valued, relying  on the strong or even on only the monotonicity of the pair \( (F,v) \), potentially eliminating the need to compute \( v^{-1} \).
 
 This motivates us to study a Generalized Proximal Point Algorithm (\textbf{GPPA}) under the monotonicity of the pair \( (F, v) \) as follows:  
\begin{equation}\label{mainal}
(\textbf{GPPA}): \quad x_0 \in \mathcal{H}, \quad x_{n+1} \in J_{\gamma_n F}^v(x_n), \quad n = 0, 1, 2, \dots
\end{equation}  
for some \( \gamma_n > 0 \), where \( J_{\gamma_n F}^v(x) = (\gamma_n F + v)^{-1} \circ v(x) \) is the \emph{warped resolvent} of \( \gamma_n F \) with kernel \( v \), introduced by Bui--Combettes in \cite{bc}. When \( v = \Id \), the identity operator, this reduces to the classical Proximal Point Algorithm (\textbf{PPA}) (see, e.g., \cite{Rockafellar}). The use of a nonlinear mapping \( v \) allows for greater flexibility and simplicity in computation, especially in cases where the identity \( \Id \) does not suffice. The mapping \( v \) can adapt to the complexity of the set-valued operator \( F \).  

The algorithm (\textbf{GPPA}) is based on the fact that \( x^* \) is a solution of \eqref{main} if and only if \( x^* \) is a solution of the fixed-point inclusion:  
\beq
x \in J_{\gamma F}^v(x), \;\gamma>0.
\eeq

In \cite{bc}, the authors employed warped resolvents to solve monotone inclusions. In this paper, our first contribution is to demonstrate that the combination of warped resolvents and the monotonicity of pairs can handle non-monotone inclusions. To enhance the effectiveness of this approach, we modify the conditions in the original definition of the warped resolvent (Definition \ref{wa}).
In \cite{bc}, it is additionally required that \( (F + v) \) be injective to ensure that \( J_F^v \) is single-valued. However, this condition is quite restrictive, as \( (F + v)^{-1} \) operates on the function \( v \) rather than directly on the variable. In our definition, this injectivity condition is unnecessary, and as a result, the warped resolvent \( J_F^v \) is generally set-valued. Nonetheless, we demonstrate that in certain important cases, \( J_F^v \) can  be still single-valued, such as under the (strong) monotonicity of the pair \( (F, v) \) (see Proposition \ref{p:warped}).  

We then show that strong convergence of \eqref{mainal} can be achieved if we find a mapping \( v \) such that \( (F, v) \) is strongly monotone (Theorem \ref{t:warped}). Additionally, if \( v^{-1} \) is well-defined, single-valued, and Lipschitz continuous, then linear convergence is obtained. In cases where only the monotonicity of \( (F, v) \) holds and \( v^{-1} \) is weakly continuous, we have the weak convergence. Note that here, we only require the properties of \( v^{-1} \), not its computation.  

Our second contribution is the introduction of the \emph{transformed resolvent} of \( F \) with respect to \( v \), defined as 
\[
T_{F}^v := v \circ (F + v)^{-1}.
\]  
When \( v = \Id \), the transformed resolvent \( T_{F}^v \) reduces to the classical resolvent of \( F \). While the warped resolvent applies \( v \) after \( (F + v)^{-1} \), the transformed resolvent applies \( v \) before \( (F + v)^{-1} \), leading to some surprising properties. Like the classical resolvent, the transformed resolvent has a key property: if the pair \( (F, v) \) is monotone and \( v^{-1} \) is not necessarily well-defined, then \( T_{F}^v \) is single-valued and firmly nonexpansive (Proposition \ref{p:transformed}).  

To highlight the effectiveness of the transformed resolvents, consider the following simple example. Let \( F = v = A \), where \( A \) is a matrix. In this case, \( (A, A) \) is monotone, and the warped resolvent \( J_F^v = (2A)^{-1}A \) can be set-valued, whereas the transformed resolvent \( T_{F}^v = A(2A)^{-1} \) is single-valued and \( \frac{1}{2} \)-Lipschitz continuous. A similar example can be constructed if \( F = A + f \), where \( A \) is a matrix and \( f \) is a small nonlinear perturbation such that \( (F, A) \) is monotone.  

Using Halpern's technique and the fact that the transformed resolvent is firmly nonexpansive (Proposition \ref{p:transformed}), we propose two algorithms, (\textbf{GPPA1}) and (\textbf{GPPA2}), which achieve weak convergence, strong  convergence and linear convergence (Theorem \ref{t:transformed} and Theorem \ref{strongc}) under more favorable conditions than those required by  (\textbf{GPPA}). Notably, the analysis of (\textbf{GPPA1}) and (\textbf{GPPA2}) only requires the injectivity of \( v \) and does not depend on the continuity of \( v^{-1} \), as in the convergence analysis of (\textbf{GPPA}). Especially even if \( v \) is not injective, we can still obtain the image of a  solution  of the original problem through the transformation $v$ (Theorem \ref{strongc}). Finally, we provide an application to quadratic programming to demonstrate the effectiveness of our approach.  

The structure of the paper is as follows. In Section \ref{sec2}, we review the concept of monotonicity for pairs of operators. In Section \ref{sec3}, we explore the properties of warped and transformed resolvents under the monotonicity of pairs. We then establish, under some mild assumptions, various results on weak convergence, strong convergence, and linear convergence of the proposed algorithms in Section \ref{sec4}. Section \ref{sec5} provides an application to quadratic programming, and we conclude the paper in Section \ref{sec6}.

\section{Notations and background material}\label{sec2}

Throughout this work, \( \mathcal{H} \) is a real Hilbert space with inner product \( \langle \cdot, \cdot \rangle \) and induced norm \( \|\cdot\| \). The set of nonnegative integers is denoted by \( \mathbb{N} \), and the set of real numbers by \( \mathbb{R} \).  

Let \( F : \mathcal{H} \rightrightarrows \mathcal{H} \) be a set-valued mapping. The \emph{domain}, \emph{range}, \emph{graph}, \emph{zeros}, and \emph{fixed points} of \( F \) are defined respectively as:  
\begin{align*}
\dom F &= \{x \in \mathcal{H} : F(x) \neq \varnothing\}, \\
\ran F &= \bigcup_{x \in \mathcal{H}} F(x), \\
\gra F &= \{(x, y) : x \in \mathcal{H}, y \in F(x)\}, \\
\zer F &= \{x \in \mathcal{H} : 0 \in F(x)\}, \\
\Fix F &= \{x \in \mathcal{H} : x \in F(x)\}.
\end{align*}  

The \emph{inverse} of \( F \) is defined by \( F^{-1}(y) = \{x \in \mathcal{H} : y \in F(x)\} \). It is easy to see that \( \dom F = \ran F^{-1} \) and \( \ran F = \dom F^{-1} \).  

The set-valued mapping \( F \) is said to be \emph{monotone} if for all \( (x, x^*), (y, y^*) \in \gra F \),
\[
\langle x^* - y^*, x - y \rangle \geq 0,
\]
and \emph{\(\alpha\)-strongly monotone} if \( \alpha \in (0, +\infty) \) and for all \( (x, x^*), (y, y^*) \in \gra F \),
\[
\langle x^* - y^*, x - y \rangle \geq \alpha \|x - y\|^2.
\]  

When \( F : \mathcal{H} \to \mathcal{H} \) is single-valued on its domain, we define:  
\begin{itemize}
\item
\( F \) is \emph{\(\ell\)-Lipschitz continuous} (\(\ell \in [0, +\infty)\)) if for all \( x, y \in \dom F \),
\[
\|F(x) - F(y)\| \leq \ell \|x - y\|;
\]
\item
\( F \) is \emph{nonexpansive} if it is \( 1 \)-Lipschitz continuous, i.e., for all \( x, y \in \dom F \),
\[
\|F(x) - F(y)\| \leq \|x - y\|;
\]
\item
\( F \) is \emph{firmly nonexpansive} (see, e.g., \cite{Rockafellar}) if for all \( x, y \in \dom F \),
\[
\|F(x) - F(y)\|^2 + \|(\Id - F)(x) - (\Id - F)(y)\|^2 \leq \|x - y\|^2,
\]
or equivalently,
\[
\|F(x) - F(y)\|^2 \leq \langle F(x) - F(y), x - y \rangle.
\]
\end{itemize}  

Next, we recall the definition of \emph{monotonicity of pairs} introduced in \cite{acl} and used in \cite{LT}, which generalizes classical monotonicity. Given set-valued mappings \( F_1, F_2 : \mathcal{H} \rightrightarrows \mathcal{H} \), we define:  
\begin{itemize}
\item
The pair \( (F_1, F_2) \) is \emph{monotone} if for all \( x, y \in \mathcal{H} \),
\beq
\langle F_1(x) - F_1(y), F_2(x) - F_2(y) \rangle \geq 0,
\eeq
i.e., for all \( x, y \in \mathcal{H} \), \( x_1^* \in F_1(x) \), \( y_1^* \in F_1(y) \), \( x_2^* \in F_2(x) \), \( y_2^* \in F_2(y) \),
$$
\langle x_1^* - y_1^*, x_2^* - y_2^* \rangle \geq 0.
$$
\item
The pair \( (F_1, F_2) \) is \emph{\(\alpha\)-strongly monotone} if \( \alpha \in (0, +\infty) \) and for all \( x, y \in \mathcal{H} \),
\beq\label{strong}
\langle F_1(x) - F_1(y), F_2(x) - F_2(y) \rangle \geq \alpha \|x - y\|^2,
\eeq
i.e., for all \( x, y \in \mathcal{H} \), \( x_1^* \in F_1(x) \), \( y_1^* \in F_1(y) \), \( x_2^* \in F_2(x) \), \( y_2^* \in F_2(y) \),
$$
\langle x_1^* - y_1^*, x_2^* - y_2^* \rangle \geq \alpha \|x - y\|^2.
$$
\end{itemize}  

It is clear that if \( F \) is monotone (resp., strongly monotone), then the pair \( (F, \Id) \) is monotone (resp., strongly monotone). The monotonicity of pairs often arises in problems involving block structures, where a natural adaptation appears, as shown in the following.
\begin{example}
Let \( F, v : \mathbb{R}^2 \to \mathbb{R}^2 \) be defined by  
\[
F(x_1, x_2) = \begin{bmatrix}
x_2 + \vert \sin(x_1) \vert \\
x_1 - \cos \vert x_2 \vert
\end{bmatrix}
\quad \text{and} \quad
v(x_1, x_2) = \begin{bmatrix}
x_2 \\
x_1
\end{bmatrix}.
\]
Then \( F \) is not monotone, but \( (F, v) \) is monotone.
\end{example}

\begin{example}
Let \( F, v : \mathbb{R}^2 \rightrightarrows \mathbb{R}^2 \) be defined by  
\[
F(x_1, x_2) = \begin{bmatrix}
\sgn(x_2) + x_1 \\
\sgn(x_1) - x_2
\end{bmatrix}
\quad \text{and} \quad
v(x_1, x_2) = \begin{bmatrix}
x_2 \\
x_1
\end{bmatrix},
\]
where  
\[
\sgn(\alpha) = \begin{cases}
1 & \text{if } \alpha > 0, \\
[-1, 1] & \text{if } \alpha = 0, \\
-1 & \text{if } \alpha < 0.
\end{cases}
\]
Then \( F \) is set-valued and non-monotone, but \( (F, v) \) is monotone.  

Next, we compute \( (F + v)^{-1} \), the warped resolvent \( J_F^v \), and the transformed resolvent \( T_{F}^v \). Given \( Y = (y_1, y_2)^\top \), we want to find \( X = (x_1, x_2)^\top \) such that \( Y \in (F + v)(X) \). This is equivalent to solving  
\[
\begin{bmatrix}
y_1 \\
y_2
\end{bmatrix}
\in
\begin{bmatrix}
\sgn(x_2) + x_1 + x_2 \\
\sgn(x_1) - x_2 + x_1
\end{bmatrix}
=
\begin{bmatrix}
\sgn(x_2) \\
\sgn(x_1)
\end{bmatrix}
+
\begin{bmatrix}
1 & 1 \\
-1 & 1
\end{bmatrix}
\begin{bmatrix}
x_2 \\
x_1
\end{bmatrix}.
\]
Let \( X_1 = (x_2, x_1)^\top \) and  
\[
F_1(X) = \begin{bmatrix}
\sgn(x_1) \\
\sgn(x_2)
\end{bmatrix}, \quad
A = \begin{bmatrix}
1 & 1 \\
-1 & 1
\end{bmatrix}.
\]
Here, \( F_1 \) is monotone and \( A \) is a positive definite matrix. We can easily compute \( X_1 = (F_1 + A)^{-1}(Y) \) and obtain \( X \).
\end{example}

\begin{example}
If \( F : \mathcal{H} \rightrightarrows \mathcal{H} \) can be decomposed as \( F = A + B \) and \( (A, B) \) is monotone, then \( (F, A) \) is monotone. This decomposition is conceptually similar to the decomposition in difference-of-convex (DC) programming. However, under the monotonicity of pairs, using our generalized proximal point algorithms, we can ensure the convergence of the generated sequences \( (x_k) \) to a solution. In contrast, using (\textbf{DCA}), we generally only achieve the convergence of \( (\|x_{k+1} - x_k\|) \) to zero while  \( (x_k) \) may not converge to a solution (see Sections \ref{sec4} and \ref{sec5}).
\end{example}

\begin{example}
Let \( F, G : \mathcal{H} \rightrightarrows \mathcal{H} \) and \( g : \mathcal{H} \to \mathcal{H} \) such that \( F \) is monotone and \( (G, g) \) is (strongly) monotone. Then \( F \circ g + G \) may not be monotone, but \( (F \circ g + G, g) \) is (strongly) monotone.
\end{example}
{Strong convergence and weak convergence of sequences are denoted by $\to$ and $\rightharpoonup$, respectively. A single-valued map $F: \mathcal{H}\to \mathcal{H}$ is said to be \emph{weakly continuous} if it is continuous with respect to the weak topology on $\mathcal{H}$, i.e., for every sequence $(x_n)_{n\in \mathbb{N}}$ in $\mathcal{H}$ with $x_n\rightharpoonup x$, one has $F(x_n)\rightharpoonup F(x)$. We say that a set-valued mapping $F: \mathcal{H} \rightrightarrows \mathcal{H}$ has a \emph{strongly-weakly closed graph} if, for every sequence $(x_n, y_n)_{n \in \mathbb{N}}$ in $\gra F$ with $x_n\rightharpoonup x$ and $y_n\to y$, it holds that $(x, y)\in \gra F$, i.e., $y \in F(x)$.}  

Finally, we recall Opial's lemma \cite{Opial}.  

\begin{lemma}[Opial's Lemma]
\label{l:Opial}
Let \( S \) be a nonempty subset of \( \mathcal{H} \), and let \( (x_n)_{n \in \mathbb{N}} \) be a sequence in \( \mathcal{H} \). Suppose that:  
\begin{enumerate}
\item
For every \( x^* \in S \), \( \lim_{n \to \infty} \|x_n - x^*\| \) exists.  
\item
Every sequential weak cluster point of \( (x_n)_{n \in \mathbb{N}} \) belongs to \( S \).  
\end{enumerate}
Then the sequence \( (x_n)_{n \in \mathbb{N}} \) converges weakly to some point \( x_\infty \in S \).  
\end{lemma}

\section{Nonlinear Resolvents}
\label{sec3}

In this section, we focus on two kinds of nonlinear resolvents. Indeed we introduce the transformed resolvent and recall the concept of warped resolvent \cite{bc} with a small modification, which form the core of our algorithms for solving  non-monotone inclusion problems.

\begin{definition}\label{wa}
Let \( F: \mathcal{H} \rightrightarrows \mathcal{H} \) and \( v: \mathcal{H} \to \mathcal{H} \) with \( \dom v = \mathcal{H} \).  
The \emph{transformed resolvent} of \( F \) with respect to \( v \) is defined as
\[
T_{F}^v = v \circ (F + v)^{-1}.
\]  
The \emph{warped resolvent} of \( F \) with kernel \( v \) is defined as  
\[
J_F^v = (F + v)^{-1} \circ v,
\]  
provided that \( \ran v \subseteq \ran(F + v) \).
\end{definition}

From the definition, it follows that \( \dom T_{F}^v = \ran(F + v) \),  $\ran T_{F}^v\subseteq \ran v$ and if \( \ran v \subseteq \ran(F + v) \), then \( \dom J_F^v = \dom v =\mathcal{H} \). See \cite[Theorem~2.3]{BWY10} for cases where \( \ran(F + v) = \mathcal{H} \), and \cite[Proposition~3.9(i)]{bc} for cases where \( \ran v \subseteq \ran(F + v) \).

\begin{proposition}
\label{p:Fix}
Let \( F: \mathcal{H} \rightrightarrows \mathcal{H} \), \( v: \mathcal{H} \to \mathcal{H} \), and \( \gamma \in (0, +\infty) \). Then the following hold:
\begin{enumerate}
\item\label{p:Fix_transformed}
\( \Fix T_{\gamma F}^v = v(\zer F) \).
\item\label{p:Fix_warped}
\( \Fix J_{\gamma F}^v = \zer F \).
\item\label{p:Fix_two}
\( \Fix T_{\gamma F}^v = v(\Fix J_{\gamma F}^v) \).
\end{enumerate}
\end{proposition}

\begin{proof}
\ref{p:Fix_transformed}: We have that
\begin{align*}
x \in \Fix T_{\gamma F}^v &\iff x \in T_{\gamma F}^v(x) \\
&\iff \exists z \in (\gamma F + v)^{-1}(x), \quad x = v(z) \\
&\iff \exists z \in \mathcal{H}, \quad x \in \gamma F(z) + v(z) \text{~and~} x = v(z) \\
&\iff \exists z \in \zer F, \quad x = v(z) \\
&\iff x \in v(\zer F).
\end{align*}

\ref{p:Fix_warped}: This follows from \cite[Proposition~3.10(i)]{bc}.  

\ref{p:Fix_two}: This is a direct consequence of \ref{p:Fix_transformed} and \ref{p:Fix_warped}.
\end{proof}

\begin{remark}
From Proposition \ref{p:Fix}, to find a zero point of \( F \), we can find a fixed point of either the transformed resolvent \( T_{\gamma F}^v \) or the warped resolvent \( J_{\gamma F}^v \). One of our contributions is to show that, under the monotonicity of \( (F, v) \), the transformed resolvent \( T_{\gamma F}^v \) is single-valued and firmly nonexpansive. Thus, we can compute the fixed points of \( T_{\gamma F}^v \) and \( J_{\gamma F}^v \) since they are related; see Proposition~\ref{p:Fix}\ref{p:Fix_two}.
\end{remark}

\begin{proposition}
\label{p:transformed}
Let \( F: \mathcal{H} \rightrightarrows \mathcal{H} \), \( v: \mathcal{H} \to \mathcal{H} \), and \( \gamma \in (0, +\infty) \). Then the following hold:
\begin{enumerate}
\item\label{p:transformed_firm}
If \( (F, v) \) is monotone, then \( T_{\gamma F}^v \) is single-valued and firmly nonexpansive.
\item\label{p:transformed_cont}
If \( (F, v) \) is \( \alpha \)-strongly monotone and \( v \) is \( L \)-Lipschitz continuous, then \( T_{\gamma F}^v \) is Lipschitz continuous with constant \( \frac{1}{1 + \alpha\gamma L^{-2}} < 1 \), and hence \( T_{\gamma F}^v \) is a contraction.
\end{enumerate}
\end{proposition}

\begin{proof}
Let \( x_1, x_2 \in \dom T_{\gamma F}^v = \ran(\gamma F + v) \). For each \( i \in \{1, 2\} \), let \( y_i \in T_{\gamma F}^v(x_i) \). Then there exists \( z_i \in (\gamma F + v)^{-1}(x_i) \) such that \( y_i = v(z_i) \). Thus, for \( i \in \{1, 2\} \),
\begin{equation}\label{eq:Fz}
x_i - v(z_i) \in \gamma F(z_i).
\end{equation}

\ref{p:transformed_firm}: By the monotonicity of \( (F, v) \),
\[
\langle (x_1 - x_2) - (v(z_1) - v(z_2)), v(z_1) - v(z_2) \rangle \geq 0,
\]
which yields
\begin{equation}\label{trainq}
\|y_1 - y_2\|^2 \leq \langle x_1 - x_2, y_1 - y_2 \rangle.
\end{equation}

It follows that \( T_{\gamma F}^v \) is single-valued and firmly nonexpansive.

\ref{p:transformed_cont}: From \eqref{eq:Fz}, using the strong monotonicity of \( (F, v) \) and the Lipschitz continuity of \( v \),
\[
\langle (x_1 - x_2) - (v(z_1) - v(z_2)), v(z_1) - v(z_2) \rangle \geq \alpha\gamma \|z_1 - z_2\|^2 \geq \alpha\gamma L^{-2} \|v(z_1) - v(z_2)\|^2.
\]
Thus,
\[
(1 + \alpha\gamma L^{-2})\|y_1 - y_2\|^2 \leq \langle x_1 - x_2, y_1 - y_2 \rangle \leq \|x_1 - x_2\| \|y_1 - y_2\|,
\]
which implies that
\[
\|y_1 - y_2\| \leq \frac{1}{1 + \alpha\gamma L^{-2}} \|x_1 - x_2\|,
\]
and the conclusion follows.
\end{proof}

\begin{remark}
From \eqref{strong}, it follows that if \( (F, v) \) is strongly monotone, then \( v \) is injective. 
\end{remark}
\begin{proposition}
\label{p:warped}
Let \( F: \mathcal{H} \rightrightarrows \mathcal{H} \), \( v: \mathcal{H} \to \mathcal{H} \), and \( \gamma \in (0, +\infty) \). Suppose that \( \ran v \subseteq \ran(\gamma F + v) \).
\begin{enumerate}

\item\label{p:warped_mono} If \( (F, v) \) is monotone and \( v \) is injective, then \( \dom J_{\gamma F}^v = \mathcal{H} \) and \( J_{\gamma F}^v \) is single-valued. 
\item\label{p:warped_strong}
If \( (F, v) \) is \( \alpha \)-strongly monotone  {then, for all $(x_1, y_1), (x_2, y_2)\in \gra J_{\gamma F}^v$, we have
\begin{align}\label{eq:y1y2}
\alpha \gamma \|y_1 - y_2\|^2 + \|v(y_1) - v(y_2)\|^2 \leq \langle v(x_1) - v(x_2), v(y_1) - v(y_2) \rangle.
\end{align}
 Furthermore if $v$ is  $L$-Lipschitz continuous then
\begin{align}\label{eq:v}
\|v(y_1) -v(y_2)\|\leq \frac{1}{1 +\alpha\gamma L^{-2}}\|v(x_1) -v(x_2)\|.
\end{align}   
If, in addition, $v^{-1}$ is well-defined $\ell$-Lipschitz continuous, then $J_{\gamma F}^v$ is $\frac{\ell L}{1 +\alpha\gamma L^{-2}}$-Lipschitz continuous.}
\end{enumerate}
\end{proposition}
\begin{proof}
\ref{p:warped_mono}:  Since \( \ran v \subseteq \ran(\gamma F + v) \), it follows that \( \dom J_{\gamma F}^v = \mathcal{H} \). {For each $i\in \{1, 2\}$, let \( x_i \in \mathcal{H} \)} and \( y_i \in J_{\gamma F}^v(x_i) = (\gamma F + v)^{-1} \circ v(x_i) \). Then \( v(x_i) \in \gamma F(y_i) + v(y_i) \), which yields
\[
v(x_i) - v(y_i) \in \gamma F(y_i).
\]

 Since  the monotonicity of \( (F, v) \), we have
\[
\langle v(x_1) - v(y_1) - v(x_2) + v(y_2), v(y_1) - v(y_2) \rangle \geq 0,
\]
or equivalently,
\begin{align}\label{eq:y1y2}
 \|v(y_1) - v(y_2)\|^2 \leq \langle v(x_1) - v(x_2), v(y_1) - v(y_2) \rangle.
\end{align}
{In view of \eqref{eq:y1y2}, $x_1 = x_2$} implies \( y_1 = y_2 \). Thus, \( J_{\gamma F}^v \) is single-valued.

\ref{p:warped_strong}: Now, assume that $(F, v)$ is $\alpha$-strongly monotone and $v$ is $L$-Lipschitz continuous. Similar to  \eqref{eq:y1y2}, one has
\begin{align*}
\alpha \gamma\| y_1-y_2 \|^2+ \|v(y_1) - v(y_2)\|^2 \leq \langle v(x_1) - v(x_2), v(y_1) - v(y_2) \rangle.
\end{align*}
If $v$ is  $L$-Lipschitz continuous,  using Cauchy--Schwarz inequality, we have 
\begin{align*}
(1 +\alpha\gamma L^{-2})\|v(y_1) - v(y_2)\|^2 \leq \|v(x_1) - v(x_2)\|\|v(y_1) - v(y_2)\|,
\end{align*}
which implies \eqref{eq:v}.

If $v^{-1}$ is well-defined and $\ell$-Lipschitz continuous, then 
\begin{align*}
\|y_1 - y_2\| =\|v^{-1}(v(y_1)) - v^{-1}(v(y_1))\|\leq \ell\|v(y_1) - v(y_2)\|,
\end{align*}
which together with \eqref{eq:v} and the Lipschitz continuity of $v$ completes the proof.
\end{proof}
\begin{remark}
Under the monotonicity of $(F,v)$, the  transformed resolvent \( T_{\gamma F}^v \) is  firmly nonexpansive while the warped resolvent \( J_{\gamma F}^v \) requires many additional assumptions to be Lipschitz continuous with big Lipschitz constant in general. 
\end{remark}

\section{Main Results}\label{sec4}

In this section, we analyze the convergence rate of some generalized proximal point algorithms (GPPAs) using warped resolvents and transformed resolvents under the monotonicity of pairs of operators to solve the non-monotone inclusion \eqref{main}. First, we propose the following assumptions.

\noindent \textbf{Assumption 1}: There exists an injective single-valued mapping \( v: \mathcal{H} \to \mathcal{H} \) such that \( (F, v) \) is monotone  and \( \ran v \subset \ran (\gamma_n F + v) \) where $(\gamma_n)_{n\in \mathbb{N}}$ is a sequence of positive real numbers.

\noindent \textbf{Assumption 1'}: There exists a single-valued mapping \( v: \mathcal{H} \to \mathcal{H} \) such that  \( (F, v) \) is \( \alpha \)-strongly monotone and \( \ran v \subset \ran (\gamma_n F + v) \) where $(\gamma_n)_{n\in \mathbb{N}}$ is a sequence of positive real numbers.

\begin{remark}\label{inj}
It is known that even if \( F \) is strongly monotone, the inclusion \( 0 \in F(x) \) may have no solution. For example, take \( F: \mathbb{R} \to \mathbb{R} \) defined by:
\[
F(x) = x + 0.5 +
\begin{cases}
1 & \text{if } x \geq 0, \\
-1 & \text{if } x < 0.
\end{cases}
\]
Then \( F \) and the pair \( (F, \Id) \) are 1-strongly monotone, but the inclusion \( 0 \in F(x) \) has no solution.
\end{remark}

\begin{theorem}
\label{t:warped}
 Suppose that Assumption 1 holds. Assume that $\zer F\neq \varnothing$ and, for all $n\in \mathbb{N}$, $\ran v \subseteq \ran(\gamma_n F + v)$. Consider the algorithm  
\begin{align}
(\textbf{GPPA}): \quad x_{n+1} = J_{\gamma_n F}^v(x_n), x_0\in \mathcal{H}, n\ge 0.
\end{align}
\begin{enumerate}
\item\label{t:warped_weakcvg}
If $F$ has a strongly-weakly closed graph, $v^{-1}$ is  well-defined weakly continuous, and $\sum_{n=0}^{+\infty} \gamma_n^2 =+\infty$ then the sequence $(x_n)_{n\in \mathbb{N}}$ converges weakly to a solution of \eqref{main}.
\item\label{t:warped_strongcvg} 
If  Assumption 1' holds and $\gamma:=\inf_{n\in \mathbb{N}} \gamma_n > 0$ then the sequence $(x_n)_{n\in \mathbb{N}}$ converges strongly to the unique solution $x^*$ of \eqref{main}. Moreover,
\begin{enumerate}
\item\label{t:warped_strongcvg_v}
If $v$ is $L$-Lipschitz continuous, then $(v(x_n))_{n\in \mathbb{N}}$ converges to $v(x^*)$ with linear rate $\frac{1}{1 + \alpha \gamma L^{-2}}$.
\item\label{t:warped_strongcvg_x}
If $v$ is $L$-Lipschitz continuous and $v^{-1}$ is  well-defined $\ell$-Lipschitz continuous, then $(x_n)_{n\in \mathbb{N}}$ converges to $x^*$ with linear rate.
\end{enumerate} 
\end{enumerate}
\end{theorem}
\begin{proof}

\ref{t:warped_weakcvg}: Let $n\in \mathbb{N}$ and $w^* \in v(\zer F)$. Then $w^* = v(x^*)$ for some $x^*\in \zer F$. By Proposition~\ref{p:Fix}\ref{p:Fix_warped}, $x^*\in \Fix J_{\gamma_n F}^v$, i.e., $(x^*, x^*)\in \gra J_{\gamma_n F}^v$. Since $(x_{n+1}, x_n)\in \gra J_{\gamma_n F}^v$,  similarly to \eqref{eq:y1y2}, we have
\baqn \nonumber
 \|v(x_{n+1}) - v(x^*)\|^2 &\leq& \langle v(x_n) - v(x^*), v(x_{n+1}) - v(x^*) \rangle\\
&=&\frac{1}{2}(\|v(x_n) -v(x^*)\|^2+ \|v(x_{n+1}) - v(x^*)\|^2-\|v(x_n) -v(x_{n+1})\|^2) 
\eaqn
which implies that
\begin{align}\label{eq:vv*}
\|v(x_{n+1}) -v(x^*)\|^2\leq \|v(x_n) -v(x^*)\|^2 -\|v(x_n) -v(x_{n+1})\|^2.
\end{align}
Thus the sequence $(\|v(x_n) - v(x^*)\|)_{n\in \mathbb{N}}$ is nonincreasing and hence convergent.

For each $n\in \mathbb{N}$, set $u_n =\frac{v(x_n) - v(x_{n+1})}{\gamma_n}$. Then $u_n\in F(x_{n+1})$ since $x_{n+1} = J_{\gamma_n F}^v(x_n)$. By telescoping \eqref{eq:vv*},
\begin{align*}
\sum_{n=0}^{+\infty} \gamma_n^2\|u_n\|^2 = \sum_{n=0}^{+\infty} \|v(x_n) -v(x_{n+1})\|^2 \leq \|v(x_0) -v(x^*)\|^2 -\lim_{n\to +\infty} \|v(x_n) -v(x^*)\|^2 < +\infty.
\end{align*}
Since $\sum_{n=0}^{+\infty} \gamma_n^2 =+\infty$, it follows that
\begin{align}\label{eq:liminf u}
\liminf_{n\to +\infty} \|u_n\| =0.
\end{align}

On the other hand, since $u_n\in F(x_{n+1})$ and $u_{n+1}\in F(x_{n+2})$, the monotonicity of $(F, v)$ implies that
\begin{align*}
0 \leq \langle u_n - u_{n+1}, v(x_{n+1}) - v(x_{n+2}) \rangle =\langle u_n - u_{n+1}, \gamma_{n+1}u_{n+1} \rangle \leq \gamma_{n+1}(\|u_n\|\|u_{n+1}\| - \|u_{n+1}\|^2),  
\end{align*}
and so $\|u_{n+1}\|^2 \leq \|u_n\|\|u_{n+1}\|$. Thus, $(\|u_n\|)_{n\in \mathbb{N}}$ is convergent. By combining with \eqref{eq:liminf u}, $\|u_n\|\to 0$ as $n\to +\infty$.

Let $\bar{w}$ be a weak cluster point of $(v(x_n))_{n\in \mathbb{N}}$. Then there exists a subsequence $(v(x_{k_n}))_{n\in \mathbb{N}}$ that converges weakly to $\bar{w}$. Since $v^{-1}$ is well-defined and weakly continuous, $x_{k_n}\rightharpoonup \bar{x} = v^{-1}(\bar{w})$ as $n\to +\infty$. Since $u_{k_n-1} \in F(x_{k_n})$ and $F$ has a strongly-weakly closed graph, we deduce that $0\in F(\bar{x})$, and $\bar{w} =v(\bar{x})\in v(\zer F)$. Using Lemma~\ref{l:Opial}, $(v(x_n))_{n\in \mathbb{N}}$ converges weakly to a point $\tilde{w}\in v(\zer F)$. Consequently, $(x_n)$ converges weakly to $\tilde{x} =v^{-1}(\tilde{w})\in \zer F$ because $v^{-1}$ is weakly continuous. 

\ref{t:warped_strongcvg}: Since $(F, v)$ is strongly monotone, from \eqref{strong} we deduce that \eqref{main} has a unique solution $x^*$. {According to Proposition~\ref{p:Fix}\ref{p:Fix_warped}, $x^*\in \Fix J_{\gamma_n F}^v$, and so $(x^*, x^*)\in \gra J_{\gamma_n F}^v$. As $(x_{n+1}, x_n)\in \gra J_{\gamma_n F}^v$, we have from Proposition~\ref{p:warped}\ref{p:warped_strong} that
\begin{align*}
\alpha \gamma \|x_{n+1} - x^*\|^2 + \|v(x_{n+1}) - v(x^*)\|^2 \leq \|v(x_n) - v(x^*)\| \|v(x_{n+1}) - v(x^*)\|,
\end{align*}
where $\gamma = \inf_{n\in \mathbb{N}} \gamma_n$.} 
Thus, $\|v(x_{n+1}) - v(x^*)\| \leq \|v(x_n) - v(x^*)\|$, and the sequence $(\|v(x_n) - v(x^*)\|)_{n\in \mathbb{N}}$ is convergent. Additionally,
\[
\alpha \gamma \|x_{n+1} - x^*\|^2 \leq \|v(x_{n+1}) - v(x^*)\| (\|v(x_n) - v(x^*)\| - \|v(x_{n+1}) - v(x^*)\|) \to 0.
\]
Passing to the limit establishes the strong convergence of $(x_n)_{n\in \mathbb{N}}$ to $x^*$.

\ref{t:warped_strongcvg_v}: If $v$ is $L$-Lipschitz continuous, then by Proposition~\ref{p:warped}\ref{p:warped_strong},
\begin{align*}
\|v(x_{n+1}) - v(x^*)\| \leq \frac{1}{1 + \alpha \gamma_n L^{-2}} \|v(x_n) - v(x^*)\| \leq \kappa \|v(x_n) - v(x^*)\|,
\end{align*}
where $\kappa := \frac{1}{1 + \alpha \gamma L^{-2}} < 1$. Therefore, $(v(x_n))_{n\in \mathbb{N}}$ converges to $v(x^*)$ with a linear rate.

\ref{t:warped_strongcvg_x}: If $v^{-1}$ is well-defined, single-valued, and $\ell$-Lipschitz continuous, then
\begin{align*}
\|x_n - x^*\| \leq \ell \|v(x_n) - v(x^*)\| \leq \ell \kappa^n \|v(x_0) - v(x^*)\| \leq \ell L \kappa^n \|x_0 - x^*\|,
\end{align*}
and the conclusion follows.
\end{proof}

\begin{remark}
\begin{enumerate}
\item
The convergence of \eqref{mainal} requires only the properties of \( v^{-1} \), not its explicit computation.
\item
If \( v^{-1} \) is well-defined (i.e., $v$ is  bijective) and \( (F, v) \) is monotone, we can rewrite the original inclusion \( 0 \in F(x) \) as \( 0 \in G(y) \), where \( G = F \circ v^{-1} \) is monotone (see, e.g., \cite{acl}) and \( y = v(x) \). However, the computation of the resolvent of \( \gamma F \circ v^{-1} \) is usually more complicated than \( (\gamma F + v)^{-1} \) or \( v^{-1} \). Moreover, within the framework of monotonicity of pairs, it is easier to encounter some interesting applications, as shown in Section \ref{sec5}.
\end{enumerate}
\end{remark}

\begin{corollary}\label{weakc}
Suppose that \( \mathcal{H} = \mathbb{R}^n \), $\zer F\neq \varnothing$, $\gamma_n=\gamma>0$, the pair \( (F, v) \) is monotone, \( F \) has a  closed graph, and \( v^{-1} \) is a well-defined continuous function. Then the sequence \( (x_k) \) generated by the algorithm (\textbf{GPPA}) converges to some solution \( x^* \).
\end{corollary}

\begin{theorem}
\label{t:transformed}
Suppose that \( (F, v) \) is monotone and $(\gamma_n)_{n\in \mathbb{N}}$ a sequence of positive real numbers. Assume that $\zer F\neq \varnothing$ and, for all $n\in \mathbb{N}$, \( \ran v \subset \ran (\gamma_n F + v) \). Consider the algorithm
\begin{align*}
(\textbf{GPPA1}): \quad x_{n+1} = T_{\gamma_n F}^v(x_n), x_0\in \ran v, n\ge 0.
\end{align*}
Then the following hold:
\begin{enumerate}
\item\label{t:transformed_weakcvg}
If $F$ has a strongly-weakly closed graph, $v^{-1}$ is well-defined and weakly continuous, and $\sum_{n=0}^{+\infty} \gamma_n^2 =+\infty$, then the sequence $(x_n)_{n\in \mathbb{N}}$ converges weakly to a point $\tilde{x}$ such that $v^{-1}(\tilde{x})$ is a solution of \eqref{main}.
\item\label{t:transformed_linearcvg}
If Assumption 1' holds, $v$ is $L$-Lipschitz continuous, and $\gamma=\inf_{n\in \mathbb{N}} \gamma_n > 0$, then the sequence $(x_n)_{n\in \mathbb{N}}$ converges with linear rate $\frac{1}{1 + \alpha\gamma L^{-2}}$ to a point $\tilde{x}$ such that $v^{-1}(\tilde{x})$ is a solution of \eqref{main}.
\end{enumerate} 
\end{theorem}
\begin{proof}
\ref{t:transformed_weakcvg}: The algorithm (\textbf{GPPA1}) is well-defined since $x_0\in \ran v$ and  \( \ran v \subset \ran (\gamma_n F + v) \). Let $n\in \mathbb{N}$ and $x^* \in v(\zer F)$. By Proposition~\ref{p:Fix}\ref{p:Fix_transformed}, $x^*\in \Fix T_{\gamma_n F}^v$. Since $x_{n+1} = T_{\gamma_n F}^v(x_n)$, similarly to \eqref{trainq}, one has
\begin{align}\label{eq:xx*}
\|x_{n+1} - x^*\|^2\leq \|x_n - x^*\|^2 - \|x_n - x_{n+1}\|^2,
\end{align}
and so $(\|x_n - x^*\|)_{n\in \mathbb{N}}$ is nonincreasing and convergent.

For each $n\in \mathbb{N}$, there exists $z_{n+1}\in (\gamma_n F +v)^{-1}(x_n)$ such that $x_{n+1} =v(z_{n+1})$. Then $x_n - x_{n+1} =x_n - v(z_{n+1})\in \gamma_n F(z_{n+1})$. Set $u_n =\frac{x_n - x_{n+1}}{\gamma_n}\in F(z_{n+1})$. Telescoping \eqref{eq:xx*} yields
\begin{align*}
\sum_{n=0}^{+\infty} \gamma_n^2\|u_n\|^2 = \sum_{n=0}^{+\infty} \|x_n - x_{n+1}\|^2 \leq \|x_0 -x^*\|^2 -\lim_{n\to +\infty} \|x_n - x^*\|^2 < +\infty,
\end{align*}
which together with $\sum_{n=0}^{+\infty} \gamma_n^2 =+\infty$ implies that
\begin{align}\label{eq:liminf}
\liminf_{n\to +\infty} \|u_n\| =0.
\end{align}

As $u_n\in F(z_{n+1})$ and $u_{n+1}\in F(z_{n+2})$, using the monotonicity of $(F, v)$ and then Cauchy--Schwarz inequality, we have that
\begin{align*}
0 &\leq \langle u_n - u_{n+1}, v(z_{n+1}) - v(z_{n+2}) \rangle \\
&= \langle u_n - u_{n+1}, x_{n+1} - x_{n+2} \rangle \\
&= \langle u_n - u_{n+1}, \gamma_{n+1}u_{n+1} \rangle \\
&\leq \gamma_{n+1}(\|u_n\|\|u_{n+1}\| - \|u_{n+1}\|^2).  
\end{align*}
Therefore, $\|u_{n+1}\|^2 \leq \|u_n\|\|u_{n+1}\|$, and so $(\|u_n\|)_{n\in \mathbb{N}}$ is convergent. In view of \eqref{eq:liminf}, $\|u_n\|\to 0$ as $n\to +\infty$.

Let $\bar{x}$ be a weak cluster point of $(x_n)_{n\in \mathbb{N}}$. Then there exists a subsequence $(x_{k_n})_{n\in \mathbb{N}}$ such that $x_{k_n}\rightharpoonup \bar{x}$ as $n\to +\infty$. Since $v^{-1}$ is well-defined and weakly continuous, $z_{k_n} =v^{-1}(x_{k_n})\rightharpoonup \bar{z} = v^{-1}(\bar{x})$ as $n\to +\infty$. Noting that $u_{k_n-1} \in F(z_{k_n})$, $u_{k_n-1}\to 0$ as $n\to +\infty$, and $F$ has a strongly-weakly closed graph, we obtain $0\in F(\bar{z})$, and hence $\bar{x} =v(\bar{z})\in v(\zer F)$. Now, by Lemma~\ref{l:Opial}, $(x_n)_{n\in \mathbb{N}}$ converges weakly to a point $\tilde{x}\in v(\zer F)$, which complete the proof.

\ref{t:transformed_linearcvg}: This follows from Proposition~\ref{p:transformed}\ref{p:transformed_cont} and the fact that \eqref{main} has a unique solution.
\end{proof}

%
%

\begin{theorem}\label{strongc}
Suppose that $\zer F\neq \varnothing$, the pair \( (F, v) \) is monotone, and \( \ran (\gamma F + v) = \mathcal{H} \) for some \( \gamma > 0 \). Then the sequence \( (x_k) \) generated by the algorithm:
\[
(\textbf{GPPA2}): \quad x_0 \in \mathcal{H}, \; x \in \mathcal{H}, \; x_{k+1} = \alpha_k x + (1 - \alpha_k) T_{\gamma F}^v(x_k), \; k = 1, 2, 3, \dots
\]
converges strongly to some  point $\tilde{x}$ such that $v^{-1}(\tilde{x})$ contains a solution of \eqref{main}, provided that:
\[
\lim_{n \to \infty} \alpha_n = 0, \quad \sum_{n=0}^\infty \alpha_n = \infty, \quad \sum_{n=0}^\infty \vert \alpha_{n+1} - \alpha_n \vert < \infty.
\]
\end{theorem}

\begin{proof}
Using Halpern's procedure (see, e.g., \cite{Halpern, nt, Wittmann}) and noting that the transformed resolvent \( T_{\gamma F}^v = v \circ (\gamma F + v)^{-1} \) is firmly nonexpansive (Proposition \ref{p:transformed}), we conclude that the sequence \( (x_k) \) converges strongly to a fixed point $\tilde{x}$  of \( T_{\gamma F}^v \) and the conclusion follows.
\end{proof}

\begin{remark}
The monotonicity of the pair \( (F, v) \) is a sufficient condition to ensure that the transformed resolvent \( T_{\gamma F}^v \) is nonexpansive. However, there are other cases where both \( T_{\gamma F}^v \) and \( J_{\gamma F}^v \) are nonexpansive. For example, if \( v \) is \( L \)-Lipschitz, \( (\gamma F + v)^{-1} \) is single-valued and \( \tilde{L} \)-Lipschitz, and \( \tilde{L}L \leq 1 \), the following example holds:
\[
F(x_1, x_2) = \begin{bmatrix}
\sgn(x_2) + x_2 - x_1 \\
-\sgn(x_1) - x_1 - x_2
\end{bmatrix}, \quad v(x_1, x_2) = \begin{bmatrix}
x_2 + x_1 \\
-x_1 + x_2
\end{bmatrix}, \quad \gamma = 1.
\]
In this case,
\[
T_{\gamma F}^v(y_1, y_2) = v \circ \begin{bmatrix}
(\sgn + 2\Id)^{-1}(-y_2) \\
(\sgn + 2\Id)^{-1}(y_1)
\end{bmatrix}
\]
is \( \frac{\sqrt{2}}{2} \)-Lipschitz continuous. Thus, \( T_{\gamma F}^v \) is a contraction, and so is \( J_{\gamma F}^v \). In general, if \( J_{\gamma F}^v = (\gamma F + v)^{-1} \circ v \) is set-valued, it may be preferable to use \( T_{\gamma F}^v \), as discussed in the introduction.
\end{remark}

%

\begin{remark}
Our result is novel, even when \( F = f \) is a single-valued mapping. In fact, the algorithm in \cite{acl} differs and it requires strong monotonicity assumptions for convergence. Our approach provides a new method for solving the equation \( 0 = f(x) \) without requiring the derivative of \( f \).
\end{remark}

\section{Applications}\label{sec5}
\subsection{Quadratic Programming Subject to Linear Constraints}
First let us consider the  problem of minimizing a nonconvex quadratic function subject to a linear constraint as follows
\beq
\min_{y\in \R^{n_1}} \frac{1}{2}y^\top Qy+c^\top y
\eeq
subject to
$$
Cy=d,
$$
where $Q=Q^\top\in \R^{n_1\times n_1}, C\in \R^{n_2\times n_1}$  are matrices, $c$ and  $y$ are $n_1 \times 1$ vectors and $d$ is an $n_2 \times 1$ vector.

We can apply the method of Lagrange multipliers. Let
$$
L(y,\lambda )= \frac{1}{2}y^\top Qy+c^\top y+\lambda^\top(Cy-d),
$$
be the Lagrangian where the Lagrange multiplier $\lambda$ is an $n_2 \times 1$ vector. By solving 
$$\left\{
\begin{array}{lll}
\frac{\partial L}{\partial y}=Qy+c+C^\top\lambda=0\\
&&\\
\frac{\partial L}{\partial \lambda}=Cy-d=0,\\
\end{array}\right. 
$$
we have 
$$
\left[ \begin{array}{ccc}
Q_{n_1\times n_1}\;\;\; \;C^\top_{n_1\times n_2}\\ \\
C_{n_2\times n_1}\;\;\;\; 0_{n_2\times n_2}
\end{array} \right] \left[ \begin{array}{ccc}
y_{n_1\times 1}\\ \\
\lambda_{n_2\times 1}
\end{array} \right]=\left[ \begin{array}{ccc}
-c_{n_1\times 1}\\ \\
d_{n_2\times 1}
\end{array} \right]
$$
which is equivalent to solving 
$$ Ax = b, $$
 where $$ A= \left[ \begin{array}{ccc}
Q_{n_1\times n_1}\;\;\; \;C^\top_{n_1\times n_2}\\ \\
C_{n_2\times n_1}\;\;\;\; 0_{n_2\times n_2}
\end{array} \right] =A^\top\in \mathbb{R}^{n \times n},\;\; n=n_1+n_2 $$
$$
x=\left[ \begin{array}{ccc}
y_{n_1\times 1}\\ \\
\lambda_{n_2\times 1}
\end{array} \right]\in \mathbb{R}^{n \times 1}, b=\left[ \begin{array}{ccc}
-c_{n_1\times 1}\\ \\
d_{n_2\times 1}
\end{array} \right]
\in \mathbb{R}^{n \times 1}. 
$$
We are interested in the case where $A$ is a singular matrix. If \( A \) is monotone, methods such as the Nesterov accelerated algorithm \cite{Nesterov}, the proximal point algorithm \cite{Rockafellar}, or the Tikhonov regularization \cite{Le,Tikhonov} can be used. If \( A \) is not monotone, the algorithm  (\textbf{DCA}) based on the forward-backward technique  can  be applied by decomposing \( A = B - C \), where \( B \) and \( C \) are positive definite matrices. This decomposition, though, becomes challenging if the Lipschitz constant of \( A \) is large.  Besides using the inverse approach, other methods such as Gauss/Gauss-Jordan elimination, LU factorization, Eigenvalue Decomposition and Singular Value Decomposition can be also applied}. In this study, we propose a straightforward alternative approach using the concept of  monotonicity of pairs and the forward-backward technique.
\vskip 2mm
  Define \( F(x) = Ax - b \) and let \( |\alpha| \) be the smallest absolute value of the non-zero eigenvalues of \( A \). Assume that \( |\alpha| \) is not too small and take \( 0 < \kappa < |\alpha|/2 \). Then \( \kappa \) is the smallest absolute value of the eigenvalues of the invertible matrix \( A + \kappa \, \Id \). Thus $\kappa^2$ is the smallest eigenvalue of the positive definite matrix \( (A + \kappa \, \Id)^2 \). Let \( v = A + 2\kappa \, \Id \) then $v$ is invertible and we have the following. 
   \begin{lemma} \label{mpair}
 If $A$ is symmetric,  then the pair \( (F,v) \) is monotone. Equivalently, the matrices $A(A + 2\kappa \, \Id)$ and $(A + 2\kappa \, \Id)A$ are monotone.
    \end{lemma}
       \begin{proof} 
       We have
\baqn
\langle F(x_1) - F(x_2), v(x_1) - v(x_2) \rangle &=& \langle A(x_1-x_2),(A + 2\kappa \, \Id) (x_1-x_2) \\
&=&\| (A + \kappa \, \Id)(x_1 - x_2) \|^2 - \kappa^2 \| x_1 - x_2 \|^2 \\
&\ge& 0.
\eaqn
    \end{proof}
Since $v^{-1}$ is Lipschitz continuous on $\R^n$, we can apply the algorithm (\textbf{GPPA}) as follows
\begin{equation}\label{appg}
x_0 \in H, \; x_{k+1} = (2A + 2\kappa \, \Id)^{-1} (Ax_k + 2\kappa x_k + b), \quad k = 1, 2, 3, \dots
\end{equation}
This algorithm converges rapidly, even for large values of \( n \) (e.g., \( n = 400, 600, 800, 1000 \)), with random choices of \( A \), \( b \), and initial points. Assuming that the solution set is non-empty (i.e., \( b \) is within the range of \( A \)), we select \( \kappa = 0.2 \), calculate the error \( e_k = \| Ax_k - b \| \) and obtain the following data:
\begin{center}
    \begin{tabular} { | c | c | c | c |  c | }
    \hline
     No. Test &n& Number of iterations $k$ & times &  $e_k$\\
    \hline
    1 & 400 &23 & 0.14s &1.3856 x$10^{-4}$\\
    \hline
    2 & 600 &24 & 0.33s&1.0688x$10^{-4}$\\
    \hline
    3 & 800 & 24 & 0.65s& 1.4576x$10^{-4}$\\
    \hline
    4 & 1000 & 30 & 1.49s& 1.2668x$10^{-4}$\\
    \hline
    \end{tabular}
    \end{center}
    
 Here we  {used}  Matlab R2016a installed in the notebook Dell Vostro 15 3000, Core i5 - 1135G7, Ram 8 Gb. Let us compare to (\textbf{DCA}), a well-known algorithms in nonconvex optimization which also uses the inverse operators. However, the performance of (\textbf{DCA}) for this case is poor, even for small dimensions. For example, when \( n = 3 \), decomposing \( A = B - C \) with \( B = A + m\Id \) and \( C = m\Id \) for \( m = 50, 100, 500, 1000 \), the algorithm converges very slowly or fails to converge in many tests. The errors $e_k$ can be very big. This behavior is expected since, in general, (\textbf{DCA}) guarantees only the convergence of \( \|x_{k+1} - x_k\| \) to zero, but not the convergence of \( (x_k) \).
 \begin{remark}
In general Lemma \ref{mpair} is not true if $A$ is not symmetric, counterexamples can be found easily, {for example let
$$A= \left[ \begin{array}{ccc}
0\;\;\;\; \; \; \;\;0\;\;\;\; \;  \; \; \; \;0\\ \\
0\;\;\;\; \; \;\; \;1\;\;\;\; \;  \; \; \; \;2\\ \\
0\;\;\;\; -2\;\;\;\; -3
\end{array} \right], x= \left[ \begin{array}{ccc}
0\\ \\
-3\\ \\
2
\end{array} \right], \kappa=1/4.
$$
Then $A$ is not symmetric and the eigenvalue values are $\lambda_1=0, \lambda_2=\lambda_3=-1.$ We have 
$$
 \langle Ax, (A + 2\kappa \Id)x \rangle  = \langle \left[ \begin{array}{ccc}
 0\\ \\
1\\ \\
0
\end{array} \right],\left[ \begin{array}{ccc}
 0\\ \\
-1/2\\ \\
1
\end{array} \right] \rangle = -1/2 < 0.
$$}
\end{remark}
\subsection{Minimizing \( \|Ax - b\|^2 \) Without Using \( A^\top A \)}

Next, we demonstrate that Algorithm (\ref{appg}) can also minimize the function \( f(x) = \|Ax - b\|^2 \), where \( b \in \mathbb{R}^{n \times 1} \) is not necessarily in the range of  the symmetric matrix  \( A \in \mathbb{R}^{n \times n} \), without computing \( A^\top A=A^2 \) as is typically required. This problem is equivalent to finding the projection of \( b \) onto the range of \( A \). Note that \( f \) is convex with gradient \( \nabla f(x) = A^\top Ax - A^\top b =A^2x-Ab\), and the Lipschitz constant of \( \nabla f \) is \( \|A^2\| \). We establish the following results.

\begin{lemma}
The point \( x^* \) is a minimum of \( f \) if and only if \( A^2x^* = Ab \).
\end{lemma}
\begin{proof}
By solving \( \nabla f = 0 \), the conclusion follows.
\end{proof}

\begin{theorem}
If \( (x_k) \) is the sequence generated by Algorithm (\ref{appg}), then \( \|A^2x_k - Ab\| \to 0 \).
\end{theorem}

\begin{proof}
From (\ref{appg}), we have:
\baq\label{mini}
2Ax_{k+1} + 2\kappa x_{k+1} = Ax_k + 2\kappa x_k + b \iff Ax_{k+1} + (A + 2\kappa\Id)(x_{k+1} - x_k) = b.
\eaq
Thus:
\[
-\langle Ax_{k+1} - b, A(x_{k+1} - x_k) \rangle = \langle (A + 2\kappa \Id)(x_{k+1} - x_k), A(x_{k+1} - x_k) \rangle \geq 0,
\]
since \( A(A + 2\kappa\Id) \) is monotone (Lemma \ref{mpair}). Therefore:
\[
\|Ax_{k+1} - b\|^2 \leq \langle Ax_{k+1} - b, Ax_k - b \rangle = \frac{1}{2}(\|Ax_{k+1} - b\|^2 + \|Ax_k - b\|^2) - \|A(x_{k+1} - x_k)\|^2.
\]
From this inequality, we deduce that the sequence \( (\|Ax_{k+1} - b\|)_{n \in \mathbb{N}} \) is decreasing and \( \|A(x_{k+1} - x_k)\| \to 0 \). Thus:
\[
 A(A + 2\kappa \Id)(x_{k+1} - x_k) \to 0.
\]
From (\ref{mini}), we deduce that \( \|A^2x_k - Ab\| \to 0 \), and the conclusion follows.
\end{proof}

\section{Conclusion} \label{sec6}

In this paper, we have demonstrated that the monotonicity of pairs of operators can be effectively used to solve non-convex optimization problems. Specifically, we proposed several Generalized Proximal Point Algorithms (\textbf{GPPAs}) to address non-monotone inclusions using warped resolvents and transformed resolvents.  {These algorithms are simple and examples show that like  (\textbf{DCA}), they are suitable for large-scale optimization while having better convergence properties, thus supporting the potential of our approach.}
However, designing an explicit procedure to select the associated function \( v \) to ensure the monotonicity of pairs, as in the quadratic application, remains an open and intriguing question for future research.


\end{document}